\documentclass[11pt,a4paper]{article} 

\usepackage{amsmath}
\usepackage{amssymb}
\usepackage{amsfonts}
\usepackage{amsthm}
\usepackage{pifont}
\usepackage{pifont}
\usepackage{graphicx}

\newtheorem*{theorem}{Theorem}

\newtheorem*{remark}{Remark}

\newcommand{\beq}{\begin{equation*}}
\newcommand{\eeq}{\end{equation*}}
\newcommand{\beqn}{\begin{equation}}
\newcommand{\eeqn}{\end{equation}}
\newcommand{\R}{\mathcal{R}_d}
\newcommand{\C}{\mathcal{C}}
\newcommand{\dd}{\mathrm{d}}
\newcommand{\T}{\mathcal{N}_{a,\,b}}
\newcommand{\EW}{\mathbb{E}}
\newcommand{\oc}{\overline{c}}

\begin{document}

\title{Buffon's problem with a pivot needle} 
\author{Uwe B\"asel}
\date{}
\maketitle

\begin{abstract}
\noindent In this paper, we solve Buffon's needle problem for a needle consisting of two line segments connected in a pivot point.
\\[0.2cm]
\end{abstract}

\section{Introduction}

The classical Buffon needle problem asks for the probability that a needle of length $\ell$ thrown at random onto a plane lattice $\mathcal{R}_d$ of parallel lines at a distance $d\geq\ell$ apart will hit one of these lines. This problem was stated and solved by Buffon in his {\em Essai d'Arithm\'etique Morale}, 1777 (see e.\:g. \cite[pp.\:71-72]{Santalo}, \cite[pp.\:501-502]{Seneta}). If an arbitrary convex body $\C$ with maximum width $\leq d$ is used in this experiment, then the hitting probability is given by $u/(\pi d)$, where $u$ denotes the perimeter of $\C$. This is the result of Barbier in 1860 \cite[pp.\:274-275]{Barbier}, \cite[p.\:507]{Seneta}. If $\C$ is a needle (line segment), then $u=2\ell$. If $\C$ is an ellipse, then there are elliptic integrals in the formulas of the hitting probabilities, see Duma and Stoka \cite{Duma_Stoka}.

We consider a needle $\T$ consisting of two line segments $C'A'$, $C'B'$ of lengths $a:=|C'A'|$ and $b:=|C'B'|$, connected in a pivot point~$C'$ (see Fig.\:\ref{fig1}), and assume $a+b\leq d$. The {\em random throw of $\T$ onto $\R$} is defined as follows: The $y$-coordinate of the point~$C'$  is a random variable uniformly distributed in $[0,d]$. The angles $\alpha$ and $\beta$ between the lines of $\R$, and segments $C'A'$ and $C'B'$, respectively, are random variables uniformly distributed in $[0,2\pi]$. All three random variables are stochastically independent.

\begin{figure}[h]
  \vspace{0cm}
  \begin{center}
    \includegraphics[scale=0.9]{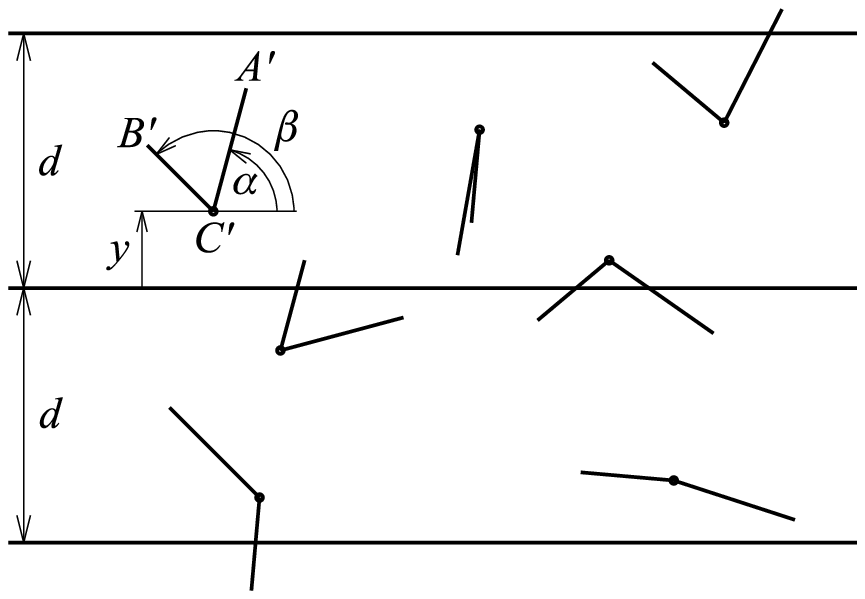}
  \end{center}
  \vspace{-0.5cm}
  \caption{\label{fig1} Lattice $\R$ and randomly thrown needle $\T$}
\end{figure}

The probability of the event that $\T$ hits two lines of $\R$ at the same time is equal to zero, even in the case $a+b=d$. The expectation $\EW(n)$ of the random variable $n=\;${\em number of intersection points between $\T$ and $\R$} is given by $\EW(n)=2(a+b)/(\pi d)$, cp.\:\cite{Ramaley}.

Here we are asking for the probabilities $p(i)$, $i\in\{0,1,2\}$, of the events that $\T$ hits $\R$ in exactly $i$ points. We denote by $A$ and $B$ the events that segments $C'A'$ and $C'B'$, respectively, hit one line of $\R$.    

\section{Hitting probabilities}

\begin{theorem}
If $a+b\leq d$, then the probabilities $p(i)$ that $\T$ hits $\R$ in exactly $i$ points are given by
\begin{align*}
  p(0) = {} & 1-\frac{(a+b)(\pi+2E(k))}{\pi^2 d}\,,\quad
  p(1) = \frac{4(a+b)E(k)}{\pi^2 d}\,,\\[0.15cm]
  p(2) = {} & \frac{(a+b)(\pi-2E(k))}{\pi^2 d}\,,
\end{align*}  
where
\beq
  E(k)=E(\pi/2,k)
	=\int_0^{\pi/2}\sqrt{1-k^2\sin^2\theta}\;\dd\theta
\eeq
is the complete elliptic integral of the second kind with $k^2=4ab/(a+b)^2$.
\end{theorem}

\begin{proof}
We observe that the angle $\phi:=\measuredangle(C'A',\,C'B')$ is a random variable uniformly distributed in $[0,2\pi]$. Due to the result of Barbier, the conditional probability $P(A\cup B\,|\,\phi)$ of $A\cup B$ for fixed value of $\phi\in[0,2\pi]$ is given by $u(\phi)/(\pi d)$, where $u(\phi)$ is the perimeter of the convex hull of $\T$. ($\T$ hits $\R$ if and only if its convex hull hits $\R$.) Using the law of total probability, the probability that $\T$ hits $\R$ is given by
\begin{align*}
  P(A\cup B) 
  = {} & \int_0^{2\pi}P(A\cup B\,|\,\phi)\:\frac{\dd\phi}{2\pi}
  = \frac{1}{2\pi^2d}\int_0^{2\pi}u(\phi)\,\dd\phi \nonumber\\[0.1cm] 
  = {} & \frac{1}{2\pi^2d}\int_0^{2\pi}\big[a+b+c(\phi)\big]\,\dd\phi
  = \frac{a+b+\oc}{\pi d}\,,	
\end{align*}
where $c:=|A'B'|$, and
\begin{align*}
  \oc
 := {} & \frac{1}{2\pi}\int_0^{2\pi}c(\phi)\,\dd\phi
  = \frac{1}{2\pi}\int_0^{2\pi}\sqrt{a^2+b^2-2ab\cos\phi}\;\dd\phi\,.
\end{align*}
Using $\cos\phi=2\cos^2(\phi/2)-1$, we have
\begin{align*}
  \oc
  = {} & \frac{1}{2\pi}\int_0^{2\pi}\sqrt{(a+b)^2-4ab\cos^2\frac{\phi}{2}}\;\,\dd\phi\\[0.1cm]
  = {} & \frac{a+b}{2\pi}\int_0^{2\pi}\sqrt{1-\frac{4ab}{(a+b)^2}\cos^2\frac{\phi}{2}}\;\,\dd\phi\,.
\end{align*}
For abbreviation we put $k^2=4ab/(a+b)^2$. From the inequality $\sqrt{ab}\leq(a+b)/2$ between the geometric and the arithmetic mean, one finds $k^2\leq 1$, hence $0\leq k\leq 1$ with $k=1$ only for $a=b$. With the substitution $\chi=\phi/2$ we get
\begin{align*}
  \oc
	= {} & \frac{a+b}{\pi}\int_0^{\pi}\sqrt{1-k^2\cos^2\chi}\;\dd\chi
	= \frac{2(a+b)}{\pi}\int_0^{\pi/2}\sqrt{1-k^2\cos^2\chi}\;\dd\chi\\
	= {} & \frac{2(a+b)}{\pi}\int_0^{\pi/2}\sqrt{1-k^2\sin^2\chi}\;\dd\chi
	= \frac{2(a+b)E(k)}{\pi}\,.
\end{align*}
It follows that
\begin{align*}
  P(A\cup B)
	= {} & \frac{a+b+\oc}{\pi d} 
	= \frac{(a+b)(\pi+2E(k))}{\pi^2 d}\,,\\[0.15cm]
  P(A\cap B) 
	= {} & P(A)+P(B)-P(A\cup B)
	= \frac{2a}{\pi d}+\frac{2b}{\pi d}-\frac{a+b+\oc}{\pi d}\\
	= {} & \frac{a+b-\oc}{\pi d}
	= \frac{(a+b)(\pi-2E(k))}{\pi^2 d}\,,
\end{align*}
and
\begin{align*}
  p(0) = {} & 1-P(A\cup B) 
       = 1-\frac{(a+b)(\pi+2E(k))}{\pi^2 d}\,,\\[0.15cm]
  p(1) = {} & P(A\cup B)-P(A\cap B) = \frac{a+b+\oc}{\pi d}
			-\frac{a+b-\oc}{\pi d} = \frac{2\,\oc}{\pi d}\\
       = {} & \frac{4(a+b)E(k)}{\pi^2 d}\,,\\[0.15cm]
  p(2) = {} & P(A\cap B) = \frac{(a+b)(\pi-2E(k))}{\pi^2 d}\,.
\qedhere
\end{align*}
\end{proof}
\noindent This is the result from \cite[pp.\:57-58]{Baesel}. There it was obtained as special case of the more general result in Corollary 4.2 \cite[p.\:56]{Baesel}. 

\begin{remark}
{\em If the angle $\phi$ is constant, then we have
\beq
  P(A\cup B)=\frac{a+b+c}{\pi d} \quad\mbox{and}\quad
  P(A\cap B)=\frac{a+b-c}{\pi d}
\eeq
with $c=\sqrt{a^2+b^2-2ab\cos\phi}$\,. This yields
\beq
  p(0)=1-\frac{a+b+c}{\pi d} \,,\quad p(1)=\frac{2c}{\pi d} \,,\quad
  p(2)=\frac{a+b-c}{\pi d},
\eeq
see Santal\'o \cite[pp.\:77-78]{Santalo}.}
\end{remark}

\section{Special cases}
If $a=b$, we have $k=1$, $E(1)=1$, and therefore
\beq
  p(0) = 1-\frac{2a(\pi+2)}{\pi^2 d} \,,\quad 
  p(1) = \frac{8a}{\pi^2 d} \,,\quad
  p(2) = \frac{2a(\pi-2)}{\pi^2 d}\,.
\eeq
If $a\not=0$ and $b=0$, then $k=0$ and $E(0)=\pi/2$, and therefore $P(A\cup B)=P(A)=2a/(\pi d)$. This is the result of the classical Buffon needle problem.

\bigskip
\begin{center}
Uwe B\"asel\\[0.15cm]
HTWK Leipzig,\\
Fakult\"at f\"ur Maschinenbau und Energietechnik,\\ 
PF 30 11 66, 04251 Leipzig, Germany\\[0.15cm]
{\small uwe.baesel@htwk-leipzig.de}
\end{center}

\end{document}